\documentclass[12pt]{article}
\usepackage{amssymb,amsmath,epsfig,amscd,eucal,psfrag,amsthm,MnSymbol,enumerate}
\usepackage{graphicx,hyperref}

\newtheorem{theorem}{Theorem}[section]

\newtheorem{lemma}[theorem]{Lemma}
\newtheorem{proposition}[theorem]{Proposition}

\newtheorem{corollary}[theorem]{Corollary}

\newtheorem{question}[theorem]{Question}
\newtheorem{letterthm}{Theorem}
\newtheorem{lettercor}[letterthm]{Corollary}

\theoremstyle{definition}
\newtheorem{definition}[theorem]{Definition}

\theoremstyle{remark}
\newtheorem{remark}[theorem]{Remark}
\newtheorem{example}[theorem]{Example}

\newcommand{\G}{\mathcal{G}}

\newcommand{\Z}{\mathbb{Z}}

\newcommand{\wh}{\widehat}

\input xy
\xyoption{all}

\title{Profinite detection of 3-manifold decompositions}
\author{Henry Wilton\footnote{Department of Pure Mathematics and Mathematical Statistics, Centre for Mathematical Sciences, Wilberforce Road, Cambridge, CB3 0WB, United Kingdom}~~and Pavel Zalesskii\footnote{Department of Mathematics, University of Bras\'ilia, 70910-9000 Bras\'ilia, Brazil}}

\begin{document}

\maketitle

\begin{abstract}
The profinite completion of the fundamental group of a closed, orientable $3$-manifold determines the Kneser--Milnor decomposition. If $M$ is irreducible, then the profinite completion determines the Jaco--Shalen--Johannson decomposition of $M$.
\end{abstract}

When trying to distinguish two compact 3-manifolds $M,N$, in practice the easiest method is often to compute some finite quotients of their fundamental groups, and notice that there is a finite group $Q$ which is a quotient of $\pi_1M$, say, but not of $\pi_1N$. It would be very useful, both theoretically and in practice, to know that this method always works.   The set of finite quotients of a group $\Gamma$ is encoded by the \emph{profinite completion} $\wh{\Gamma}$ (the inverse limit of the system of finite quotient groups), and so one is naturally led to the following question.

\begin{question}\label{Main question}
Let $M$ be a compact, orientable 3-manifold.  To what extent is $\pi_1M$ determined by its profinite completion?
\end{question}

In particular, if $M$ is determined among all compact, orientable 3-manifolds by $\wh{\pi_1M}$, then $M$ is said to be \emph{profinitely rigid}.  If there are at most finitely many compact, orientable 3-manifolds $N$ with $\wh{\pi_1M}\cong\wh{\pi_1N}$, then $M$ is said to have \emph{finite genus}.  More precise versions of Question \ref{Main question} ask which 3-manifolds are profinitely rigid, which have finite genus and whether various properties of $M$ are determined by $\wh{\pi_1M}$.

The results of this paper show that the profinite completion $\wh{\pi_1M}$ determines both the Kneser--Milnor and the JSJ decompositions of $M$. This complements our previous results showing that $\wh{\pi_1M}$ determines the geometry of $M$ \cite{wilton_distinguishing_2017}.  The first theorem concerns the Kneser--Milnor decomposition.

\begin{letterthm}\label{introthm: KM detection}
Let $M,N$ be closed, orientable 3-manifolds with Kneser--Milnor decompositions $M\cong M_1\# \ldots \# M_m \# r(S_1\times S_2)$ and $N\cong N_1\# \ldots \# N_n \# s(S_1\times S_2)$. If $\wh{\pi_1M}\cong\wh{\pi_1N}$ then $m=n$, $r=s$, and up to re-indexing, the image of $\wh{\pi_1M}_i$ is conjugate to $\wh{\pi_1N}_i$ for each $i$.
\end{letterthm}

In particular, $\wh{\pi_1M}$ determines whether or not $M$ is irreducible.  While this work was in progress we discovered that a similar result has also been proved in the pro-$p$ setting by Wilkes, using $l^2$-cohomology \cite[Proposition 6.2.4]{wilkes_profinite_2017}.  Our proof is different, using the continuous cohomology of the profinite completion, and naturally generalizes to our next theorem, which shows that the profinite completion determines the JSJ decomposition of $M$.

\begin{letterthm}\label{introthm: JSJ detection}
Let $M$ and $N$ be closed, orientable, irreducible 3-manifolds, and suppose $\wh{\pi_1M}\cong\wh{\pi_1N}$.  Then the underlying graphs of the JSJ decompositions of $\pi_1M,\pi_1N$ are isomorphic, and corresponding vertex groups have isomorphic profinite completions.
\end{letterthm}

See Theorem \ref{thm: main} for a more precise statement, phrased in terms of \emph{profinite Bass--Serre trees}.  Partial results along the lines of Theorem \ref{introthm: JSJ detection} have also been obtained by Wilkes \cite[Theorem I]{wilkes_profinite_2017}.

In \cite[Theorem 8.4]{wilton_distinguishing_2017}, it was shown that the profinite completion of the fundamental group of a closed, orientable 3-manifold $M$ determines the geometry of $M$.  As an immediate consequence of Theorem \ref{introthm: JSJ detection}, we can extend this result to the case with toral boundary.  Recall that $\overline{H}$ denotes the closure of a subgroup $H$ in the profinite completion $\overline{G}$.

\begin{lettercor}\label{introcor: Geometry with toral boundary}
Let $M,N$ be compact, orientable, irreducible 3-manifolds with non-empty toral boundaries; let $\{P_1,\ldots,P_m\}$ and $\{Q_1,\ldots,Q_n\}$ be conjugacy representatives for the boundary subgroups of $\pi_1M$ and $\pi_1N$ respectively.  Suppose that $\wh{\pi_1M}\cong\wh{\pi_1N}$, that $m=n$, and that the isomorphism takes $\overline{P}_i$ to $\overline{Q}_i$ for each $i$. If $M$ is geometric then $N$ is also geometric, with the same geometry. In particular, $M$ is Seifert fibred if and only if $N$ is Seifert fibred.
\end{lettercor}
\begin{proof}
Let $D$ be the double of $M$ and $E$ be the double of $N$.  Note that $\wh{\pi_1D}\cong\wh{\pi_1E}$, and also that the isomorphism respects the profinite completions of the boundary subgroups of $\pi_1M$ and $\pi_1N$. The result follows from the observation that the geometry of $M$ is reflected in properties of the double $D$.

Indeed, if $M$ is Seifert fibred then so is $D$. In this case, if $M$ is homeomorphic to an interval bundle over the torus then $D$ has Euclidean geometry; otherwise, $M$ and $D$ both have $\mathbb{H}^2\times\mathbb{R}$ geometry. Finally, $M$ is hyperbolic if and only if $D$ has non-trivial JSJ decomposition, and the boundary tori of $M$ are the only JSJ tori of $D$.  Combining these facts with Theorem \ref{introthm: JSJ detection} and \cite[Theorem 8.4]{wilton_distinguishing_2017}, the result follows.
\end{proof}

In light of Theorem \ref{introthm: JSJ detection}, the next step in addressing Question \ref{Main question} is to consider the pieces of the JSJ decomposition.  The Seifert-fibred case has been resolved by Wilkes \cite{wilkes_profinite_2017c}, building on work of Hempel \cite{hempel_some_2014}: Seifert fibred 3-manifolds are not profinitely rigid, but do have finite genus, and Wilkes was able to give a complete description of when two such 3-manifold groups have isomorphic profinite completions; he was subsequently able to extend this to a complete answer to Question \ref{Main question} for graph manifolds \cite[Theorem 10.9]{wilkes_profinite_2018}.  In that paper, Sol-manifolds were not included in the class of graph manifolds. Nevertheless, Sol-manifolds are also well understood: they are not profinitely rigid \cite{funar_torus_2013}, but do have finite genus \cite{grunewald_polycyclic_1980}.  A definitive treatment of the case of Sol-manifolds would be a valuable addition to the literature.

The case of hyperbolic 3-manifolds remains an important open problem.  The complement of the figure-eight knot was shown to be  profinitely rigid by Boileau--Friedl \cite{boileau_profinite_2015} and by Bridson--Reid \cite{bridson_profinite_2015}; see also \cite{bridson_determining_2016} for analogous results for Fuchsian groups and \cite{bridson_profinite_2017} for once-punctured-torus bundles. 

Both \cite{boileau_profinite_2015} and \cite{bridson_profinite_2015}, as well as \cite{bridson_profinite_2017}, rely on results showing that fibredness is a profinite invariant in certain contexts. This has recently been proved in full generality by Jaikin-Zapirain \cite{jaikin_recognition_2017}.   Ueki also recently showed that the profinite completion of a knot group determines the Alexander polynomial of the knot \cite{ueki_profinite_2018}.

%In Section \ref{sec: Questions}, we summarize the remaining open questions about profinite rigidity of 3-manifolds.

The results of this paper are proved by considering profinite Poincar\'e Duality groups.  The main difficulty in the above theorems is to show that profinite completions of 3-manifold groups do not admit unexpected splittings which are not induced by splittings of the underlying group.  It is well known that non-splitting theorems for discrete Poincar\'e Duality groups follow from the Mayer--Vietoris sequence.  As a result of the work of Agol, Wise et al.\ on the Virtual Haken conjecture \cite{agol_virtual_2013,wise_structure_2012}, 3-manifold groups are known to be \emph{good} in the sense of Serre, meaning that the cohomology of the profinite completion is isomorphic to the ordinary cohomology (with coefficients in finite modules). Furthermore, a version of the Mayer--Vietoris sequence is known for efficient decompositions of profinite completions.   The main idea of the proofs of Theorems \ref{introthm: KM detection} and \ref{introthm: JSJ detection} is to prove the analogues for profinite completions of the non-splitting theorems from the discrete case.

\subsection*{Acknowledgements}

The authors are grateful to Peter Kropholler for explaining the results of \cite{kropholler_splittings_1988}, and to Michel Boileau and Stefan Friedl for pointing out Corollary \ref{introcor: Geometry with toral boundary}.  This work was completed while the first author was a participant of the \emph{Non-positive curvature, group actions and cohomology} programme at the Isaac Newton Institute, funded by the EPSRC Grant number EP/K032208/1. The first author is also supported by an EPSRC Standard Grant EP/L026481/1. The second author was supported by CAPES  as part of the `Estagio Senior' programme, and thanks Trinity College and the Department of Pure Mathematics and Mathematical Statistics at the University of Cambridge for their hospitality.

\section{Preliminaries on profinite groups}

\subsection{Profinite trees}

A \emph{graph} $\Gamma$ is a disjoint union $E(\Gamma) \cup V(\Gamma)$ of sets, with two maps $d_0, d_1 : \Gamma \to V(\Gamma)$ that are the identity on the set of vertices $V(\Gamma)$.  For an element $e$ of the set of edges  $E(\Gamma)$, $d_0(e) $ is called the \emph{initial} and $d_1(e) $ the \emph{terminal} vertex of $e$.

\begin{definition}
A \emph{profinite graph} $\Gamma$ is a graph such that:
\begin{enumerate}[(i)]
\item $\Gamma$ is a profinite space (i.e.\ an inverse limit of finite discrete spaces);
\item $V(\Gamma)$ is closed; and
\item the maps $d_0$ and $d_1$
are continuous.
\end{enumerate}
Note that $E(\Gamma)$ is not necessary closed.
\end{definition}
A \emph{morphism} $\alpha:\Gamma\longrightarrow \Delta$ of profinite graphs is a continuous map with $\alpha d_i=d_i\alpha$ for $i=0,1$.

By \cite[Proposition ~1.7]{zalesskii_subgroups_1988}  or \cite[Proposition 2.1.4]{ribes_profinite_2017} every profinite graph $\Gamma$ is an inverse limit of finite quotient graphs of $\Gamma$.

For a profinite space $X$  that is the inverse limit of finite discrete spaces $X_j$, $[[\widehat{\mathbb{Z}} X]]$ is defined to be the inverse limit  of $ [\widehat{\mathbb{Z}} X_j]$, where $[\widehat{\mathbb{Z}} X_j]$ is the free $\widehat{\mathbb{Z}}$-module with basis $X_j$. For a pointed profinite space $(X, *)$ that is the inverse limit of pointed finite discrete spaces $(X_j, *)$, $[[\widehat{\mathbb{Z}} (X,*)]]$ is the inverse limit  of $ [\widehat{\mathbb{Z}} (X_j, *)]$, where $[\widehat{\mathbb{Z}} (X_j, *)]$ is the free $\widehat{\mathbb{Z}}$-module with basis $X_j \setminus \{ * \}$ \cite[Chapter~5.2]{ribes_profinite_2010}.

For a profinite graph $\Gamma$ define the pointed space $(E^*(\Gamma), *)$ as  $\Gamma / V(\Gamma)$ with the image of $V(\Gamma)$ as a distinguished point $*$, and denote the image of $e\in E(\Gamma)$ by $\bar{e}$. 

\begin{definition}
A \emph{profinite tree}  $\Gamma$ is a profinite graph such that the sequence
\[
0 \to [[\widehat{\mathbb{Z}}(E^*(\Gamma), *)]] \stackrel{\delta}{\rightarrow} [[\widehat{\mathbb{Z}} V(\Gamma)]] \stackrel{\epsilon}{\rightarrow} \widehat{\mathbb{Z}} \to 0
\]
is exact, where $\delta(\bar{e}) = d_1(e) - d_0(e)$ for every $e \in E(\Gamma)$ and $\epsilon(v) = 1$ for every $v \in V(\Gamma)$.
\end{definition}

If $v$  and $w$ are elements  of a profinite tree   $T$, we denote by $[v,w]$ the smallest profinite subtree of $T$ containing $v$ and $w$ and call it a \emph{geodesic} (cf.\ \cite[1.19]{zalesskii_subgroups_1988} or \cite[Proposition 2.4.9]{ribes_profinite_2017}).

By definition a profinite group $G$ \emph{acts} on a profinite graph $\Gamma$ if  we have a continuous action of $G$ on the profinite space $\Gamma$ that commutes with the maps $d_0$ and $d_1$.

We shall need the following lemma; its proof is contained in the first eight lines of the proof of \cite[Lemma 2.3]{zalesskii_profinite_1990}.

\begin{lemma}\label{open subgroup} Suppose that a profinite group $G$ acts on a profinite tree $T$ and does not fix any vertex. Then there exists an open normal subgroup $U$ of $G$ that is not generated by its vertex stabilizers.\end{lemma} 

When we say that ${\cal G}$ is a \emph{finite graph of profinite groups} we mean that ${\cal G}$ contains the data of the
underlying finite graph, the edge profinite groups, the vertex profinite groups and the attaching continuous maps. More precisely,
let $\Delta$ be a connected finite graph. The data of a  graph of profinite groups $({\cal G},\Delta)$ over
$\Delta$ consists of a profinite group ${\cal G}(m)$ for each $m\in \Delta$, and continuous monomorphisms
$\partial_i: {\cal G}(e)\longrightarrow {\cal G}(d_i(e))$ for each edge $e\in E(\Delta)$.

The definition of the profinite fundamental  group of a connected profinite graph of profinite groups is quite involved (see \cite{zalesskii_fundamental_1989} or \cite[Chapter 6]{ribes_profinite_2017}). However, the profinite fundamental  group $\Pi_1(\G,\Gamma)$ of a finite graph of finitely generated profinite groups $(\G, \Gamma)$ can be defined as the profinite completion of the abstract (usual) fundamental group $\Pi_1^{abs}(\G,\Gamma)$ (using here that every subgroup of finite index in a finitely generated profinite group is open,   \cite[Theorem 1.1]{nikolov_finitely_2007a}).  The fundamental profinite group
$\Pi_1(\G,\Gamma)$ has the following presentation:
\begin{eqnarray}\label{presentation}
\Pi_1(\mathcal{G}, \Gamma)&=&\langle \G(v), t_e\mid rel(\G(v)),
\partial_1(g)=\partial_0(g)^{t_e}, g\in \G(e),\nonumber\\
 & &t_e=1 \  {\rm for}\ e\in
T\rangle;
\end{eqnarray}
where $T$ is a maximal subtree of $\Gamma$ and
$\partial_0:\G(e)\longrightarrow
\G(d_0(e)),\partial_1:\G(e)\longrightarrow \G(d_1(e))$ are
monomorphisms.

In contrast to the abstract case, the vertex groups of $(\G,
\Gamma)$ do not always embed in $\Pi_1(\G,\Gamma)$. If they do embed,
$(\G,\Gamma)$ is called \emph{injective}. If $(\G,\Gamma)$ is not injective the edge and
vertex groups can be replaced by their images in
$\Pi_1(\G,\Gamma)$, and after this replacement $(\G,\Gamma)$
becomes injective (see \cite[Section 6.4]{ribes_profinite_2017}).

The profinite  fundamental  group $\Pi_1(\G,\Gamma)$ acts on the standard profinite  tree $T$ (defined analogously to the abstract Bass--Serre tree) associated to it, with vertex and edge stabilizers being conjugates of vertex and edge groups, and such that $\Pi_1(\G,\Gamma)\backslash T=\Gamma$ \cite[Proposition 3.8]{zalesskii_subgroups_1988} or \cite[Theorem 6.3.5]{ribes_profinite_2017}. In particular, this applies   to the cases of an amalgamated free product $G=G_1\amalg_HG_2$ ($\Gamma$ is an edge with two vertices) and an HNN-extension $G=G_1\amalg_H$ ($\Gamma$ is a loop); if $(\G,\Gamma)$ is injective and, in the  case of an amalgamated free product, $G_1\neq H\neq G_2$,  we say that $G$ \emph{splits} over $H$.

\begin{example}\label{graph group completion}
If $G=\pi_1(\G,\Gamma)$ is the fundamental group of a finite graph of (abstract) groups then one has the induced graph of profinite completions of edge  and vertex groups $(\widehat\G,\Gamma)$ and  a natural homomorphism $G=\pi_1(\G,\Gamma)\longrightarrow \Pi_1(\widehat\G,\Gamma)$. It is an embedding if $\pi_1(\G,\Gamma)$ is residually finite. In this case $\Pi_1(\wh{\G},\Gamma)$ is simply the profinite completion  $\widehat{G}$. Moreover, if  the edge  groups $\G(e)$     are separable in $G$ then  the standard tree $T_G$  naturally embeds in the standard profinite tree $\widehat T_G$  (see \cite[Proposition 2.5]{cotton-barratt_detecting_2013}).  In particular this is the case  if edge groups are finitely generated and $G$ is subgroup separable.
\end{example}

\subsection{Profinite Poincar\'e duality groups}

In this section we collect the facts about profinite groups that we will need.    The following results are all profinite analogues of well known results in the setting of discrete groups.  Let $\Z_p$ denote the ring of $p$-adic integers.

\begin{definition}[\cite{symonds_cohomology_2000}]
Let $p$ be a prime.   A profinite group $G$ of type $p$-$FP_{\infty}$  is called a \emph{Poincar\'e duality group} at $p$ of dimension $n$ if   $cd_p(G)=n$ and
\[
\begin{array}{ll}
H^i(G,\Z_p[[G]])= 0, & \mbox{ if } i\not= n, \vspace{0.2cm}\\
H^n(G,\Z_p[[G]])\cong \Z_p & \mbox {(as abelian groups)}.
\end{array}
\]
We say that such a group $G$  is a \emph{profinite $PD^n$-group at $p$}.
\end{definition}

If $G$ is a profinite group with $cd_p(G)< \infty$ and $U$ is an open subgroup of $G$, then $G$ is a profinite $PD^n$-group at $p$ if and only if $U$ is a profinite $PD^n$-group at $p$ (see \cite[Remark 4.2.9]{symonds_cohomology_2000}).

The proofs of our main results rely on the following lemma.  In the discrete case, the corresponding result is Strebel's theorem \cite{strebel_remark_1977}.  In the profinite case, this is Exercise 5(b) on p.\ 44 of \cite{serre_galois_1997}.  The reader is referred to \cite[Section 2.3]{ribes_profinite_2010} for the definition of \emph{supernatural numbers}.

\begin{lemma}\label{lem:PD}
Let $G$ be a profinite $PD^n$ group at $p$ and $H$  a closed subgroup of $G$ such that the supernatural number $p^{\infty}$ divides $[G:H]$. Then $cd_p(H)<n$.
\end{lemma}

The following theorem is the profinite analogue of the well known fact in the discrete setting that $PD^n$ groups cannot split over groups of much smaller cohomological dimension \cite[Proposition V.7.4]{dicks_groups_1989}.

\begin{theorem}\label{thm: Profinite PDn non-splitting}
Suppose that $G$ is a profinite $PD^n$ group at every prime $p$.  If $G$ acts on a profinite tree $T$ with edge stabilizers of $cd(G_e)<n-1$, then $G$ fixes a vertex.
\end{theorem}
 
\begin{proof} By \cite[paragraph 2.7]{zalesskii_subgroups_1988},
\[
{\rm cd}\  G\leq \sup\{{\rm cd}\  G_v, {\rm cd}\ G_e+1\mid v\in V(T), e\in E(T)\}.
\] 
Suppose that $G$ acts on $T$ without fixing a vertex.  We now argue that there exists $p$ such that the supernatural number $p^{\infty}$  divides $[G:G_v]$ for every $v\in V(\Gamma)$,  and deduce a contradiction from Lemma \ref{lem:PD}. 

By \cite[Lemma 1.5]{zalesskii_profinite_1990} or \cite[Proposition 2.4.12]{ribes_profinite_2017} we may assume that the action of $G$ on $T$ is irreducible (i.e.\ does not contain proper $G$-invariant subtrees). If $K$ is the kernel of the action then $G/K$ acts faithfully on $T$.  Hence by \cite[Proposition 2.10 and Lemma 2.7]{zalesskii_profinite_1990}  or \cite[Theorem 4.2.10]{ribes_profinite_2017} $G/K$ contains a free pro-$p$ subgroup acting freely on $T$ and therefore so does $G$, whence $p^\infty|[G:G_v]$ as claimed.
\end{proof}

We will apply these results to discrete groups $\Gamma$ such that the cohomology of $\Gamma$ is closely intertwined with the cohomology of the profinite completion $\widehat{\Gamma}$ -- Serre called such groups `good' \cite[I.2.6]{serre_galois_1997}.

\begin{definition}\label{defn: Good}
A discrete group $\Gamma$ is \emph{good (in the sense of Serre)} if, for any finite $\Gamma$-module $M$, the natural map to the profinite completion $\Gamma\to \wh{\Gamma}$ induces an isomorphism $H^*(\Gamma,M)\cong H^*(\wh{\Gamma},M)$ (where the cohomology of the profinite group $\wh{\Gamma}$ is defined using the continuous $\mathrm{Hom}$ functor).
\end{definition}

It has been noticed in various places (eg. \cite{cavendish_finite_2012}, \cite{aschenbrenner_manifold_2015}; cf.\ \cite{grunewald_cohomological_2008}) that 3-manifold groups are good.  For convenience, we record the result here.

\begin{theorem}\label{thm: Goodness}
If $M$ is a closed 3-manifold then $\pi_1M$ is good.
\end{theorem}
\begin{proof}
Since goodness passes to finite extensions, we may assume that $M$ is orientable.  By \cite[Proposition 4.3]{wilton_profinite_2010} and the usual
 Kneser--Milnor and JSJ decompositions, it suffices to prove that Seifert fibred and hyperbolic 3-manifold groups are good.
 The Seifert-fibred case is Proposition 4.2 of the same paper, and the case of closed hyperbolic 3-manifolds follows from the virtually fibred theorem
 \cite{agol_virtual_2013}, by \cite[Lemmas 3.2 and 3.3]{grunewald_cohomological_2008}.
\end{proof}

The next result is the subject of \cite[Theorem 4.1]{kochloukova_profinite_2008} for $PD^3$-groups, and for general $n$ the proof can be repeated
replacing $3$ by n.

\begin{theorem}\label{thm: PDn completion }
If $\Gamma$ is a good $PD^n$ group, then $\wh{\Gamma}$ is $PD^n$  at every $p$.
\end{theorem}

We immediately obtain a profinite non-splitting result for good Poincar\'e duality groups.

\begin{corollary}\label{cor: Good PDn groups}
Let $G$ be $PD^n$ group which is good in the sense of Serre. Then any action of $\wh{G}$ on a profinite tree with edge stabilizers of cohomological dimension $n-2$ has a global fixed point.
\end{corollary}
\begin{proof}
Since $G$ is good, $\widehat G$ is a profinite $PD^n$ group at $p$ for every $p$ by Theorem \ref{thm: PDn completion }, so the result follows from Theorem \ref{thm: Profinite PDn non-splitting}.
\end{proof}

\begin{remark}
The combined hypotheses of goodness and $PD^n$ apply to many examples in dimensions 2 and 3, but are restrictive in higher dimensions.
\end{remark}

Combining all of the above results, we obtain the following fact, which will be extremely useful to us in what follows.

\begin{corollary}\label{cor: No procyclic splitting}
If $M$ is  a closed, orientable, irreducible 3-manifold then any action of $\wh{\pi_1M}$ on a profinite tree with procyclic edge stabilizers has a global fixed point.
\end{corollary}
\begin{proof}
By the Sphere Theorem, irreducible 3-manifolds either have finite fundamental group or are aspherical (see, for instance, \cite[(C.1)]{aschenbrenner_manifold_2015}).   In the first case $\wh{\pi_1M}$ is finite, and the result follows from \cite[Theorem 2.10]{zalesskii_subgroups_1988} or \cite[Theorem 4.1.8]{ribes_profinite_2017}. In the second case, $\pi_1M$ is $PD^3$, and the result follows from Theorem \ref{thm: Goodness} and Corollary \ref{cor: Good PDn groups}.
\end{proof}

\section{The Kneser--Milnor decomposition}

As a warm-up, we show that the profinite completion of a 3-manifold group determines its Kneser--Milnor decomposition.  As noted above, this result can also be obtained using methods from $l^2$-cohomology \cite{wilkes_profinite_2017}. Recall that a closed $3$-manifold $M$ is \emph{irreducible} if every embedded 2-sphere bounds a 3-ball; equivalently, $\pi_1M$ does not admit a non-trivial splitting over the trivial subgroup.

\begin{proposition}\label{prop: Profinite irreducible}
Suppose that $M_1,M_2$ are closed, orientable 3-manifolds.  If $\wh{\pi_1M}_1\cong\wh{\pi_1M}_2$ and $M_1$ is irreducible then so is $M_2$.
\end{proposition}
\begin{proof}
If $M_2$ were reducible then $\pi_1M_2$ would act on a tree with trivial edge stabilizers and without a global fixed point, and $\wh{\pi_1M}_1\cong\wh{\pi_1M}_2$ would act on a profinite tree with trivial edge stabilizers and without a global fixed point.   This contradicts Corollary \ref{cor: No procyclic splitting}.
\end{proof}

Non-irreducible 3-manifolds admit non-trivial \emph{Kneser--Milnor decompositions}.  If $M$ is a closed, oriented 3-manifold then the Kneser--Milnor decomposition decomposes $M$ as a connect sum
\[
M\cong N_1\#\ldots\#N_m\# F_r
\]
where each $N_i$ is irreducible and $F_r$ is a connect sum of copies of $S^1\times S^2$.  The $N_i$ are uniquely determined, in an appropriate sense.  In particular, the conjugacy classes of the subgroups $\pi_1N_i$ are unique up to reordering, and the integer $r$ is also unique. The reader is referred to \cite[Theorem 1.2.1]{aschenbrenner_manifold_2015} for details.

\begin{theorem}[Profinite Kneser--Milnor]
Consider closed, orientable 3-manifolds with Kneser--Milnor decompositions $M=N_1 \# \ldots \# N_m \# F_{r}$ and \sloppy $M'=N'_1 \# \ldots \# N'_{m'} \# F_{r'}$, where each $N_i$ and $N'_j$ is irreducible and $F_r$ and $F_{r'}$ are connect sums of $S^1\times S^2$'s.  If $\wh{\pi_1M}\cong\wh{\pi_1M'}$ then $m=m'$, $r=r'$, and up to reordering, $\wh{\pi_1N_i}$ is conjugate to $\wh{\pi_1N'_i}$ for each $i$.
\end{theorem}
\begin{proof}
Let $S$ be the Bass--Serre tree of the corresponding decomposition of $\pi_1M$, and let $\wh{S}$ be the corresponding profinite tree for $\wh{\pi_1M}$ on which $\pi_1M$ acts with trivial edge stabilizers. By Corollary \ref{cor: No procyclic splitting}, each profinite completion $\wh{\pi_1N'_i}$ fixes a vertex of $\wh{S}$, and hence is conjugate into some $\wh{\pi_1N_j}$.  By symmetry, each $\wh{\pi_1N_i}$ is conjugate into some $\wh{\pi_1N'_j}$.  Profinite subgroups cannot be conjugate to proper subgroups of themselves, as it would imply the same for some finite image, and for a finite group it is clearly impossible.  Therefore, it follows that  $m=m'$ and the profinite completions of the vertex groups are conjugate. Factoring the normal closures of these subgroups out, we see that $\wh{F}_r\cong\wh{F}_{r'}$ and hence $r=r'$ as claimed.
\end{proof}

\section{Cusped hyperbolic 3-manifolds}

An immediate consequence of Corollary \ref{cor: Good PDn groups} is that, for a closed $3$-manifold $M$, $\wh{\pi_1M}$ does not split over a subgroup of cohomological dimension 0 or 1  (for instance a profinite free group). In this section, we prove some profinite non-splitting results for hyperbolic manifolds with toral boundary.  In the hyperbolic case, we will need a fact from \cite{wilton_distinguishing_2017}, describing the non-procyclic abelian subgroups of $\wh{\pi_1N}$.

\begin{proposition}\label{prop: Cusp uniqueness}
Let $N$ be a finite-volume hyperbolic 3-manifold and $A$ a closed abelian subgroup of $\wh{\pi_1N}$. If $A$ is not procyclic then $A$ is in the closure of a peripheral subgroup of $\pi_1N$, and this peripheral subgroup is unique up to conjugacy.
\end{proposition}
\begin{proof}
By \cite[Theorem 9.3]{wilton_distinguishing_2017}, $A$ is conjugate into the closure of a peripheral subgroup, and by \cite[Lemma 4.5]{wilton_distinguishing_2017}, the conjugacy class of the cusp subgroup is unique.
\end{proof}

In the classical $PD^n$ setting, one handles manifolds with boundary using the theory of \emph{$PD^n$ pairs} \cite{dicks_groups_1980}.  One of the upshots of this theory is that the fundamental group of an aspherical manifold with aspherical boundary cannot split over a boundary subgroup, relative to the collection of boundary subgroups.  (This can be deduced from the results of \cite{kropholler_splittings_1988}.) No doubt the profinite analogue of this statement can be proved by developing the theory of profinite $PD^n$ pairs.  We take a quicker route here: we prove the result in the cusped hyperbolic case, using Dehn filling.  First, we need to recall the definition of an \emph{acylindrical} splitting.

\begin{definition}\label{defn: Acylindrical}
An action of a group $\Gamma$  on a tree $T$ is \emph{$k$-acylindrical} (for an integer $k$) if,  for every $\gamma\in\Gamma\smallsetminus 1$, the subtree fixed by $\gamma$ is either empty or of diameter at most $k$.  Likewise, an action of a profinite group $G$ on a profinite tree $\wh{T}$ is $k$-acylindrical if the subtree fixed by $\hat{\gamma}$ is either empty or of diameter at most $k$, for every $\hat{\gamma}\in G$.  Such an action is called \emph{acylindrical} if it is $k$-acylindrical for some $k$.
\end{definition}

Acylindricity gives useful control over non-cyclic abelian subgroups, via the following lemma. This was proved in  \cite[Theorem 5.2]{wilton_distinguishing_2017}. (The discrete version of this fact is left as an instructive exercise to the reader.)

\begin{lemma}\label{lem: Abelian acylindrical fixed point}
If $A$ is an abelian, profinite, non-procyclic group, and $A$ acts acylindrically on a profinite tree $T$, then $A$ fixes a vertex.
\end{lemma}

We are now ready to prove the non-splitting result for hyperbolic manifolds with cusps.

\begin{lemma}\label{lem: non-splitting over cusp}
If $N$ is a compact, orientable, hyperbolic 3-manifold with toral boundary and  $\wh{\pi_1N}$  acts on a profinite tree $T$ with each edge stabilizer either procyclic or conjugate into  a peripheral subgroup, then $\wh{\pi_1N}$ fixes  a vertex.
\end{lemma}
\begin{proof}
First, note that if $\wh{\pi_1N}$ acts on  $T$ without fixed points then, by Lemma \ref{open subgroup}, after passing to a proper open subgroup we may assume that $\wh{\pi_1N}$ is not generated by vertex stabilizers. 

Let the family of peripheral subgroups of $\pi_1N$ be $P_1,\ldots,P_n$.  By Thurston's hyperbolic Dehn filling theorem (see, for instance, \cite{agol_bounds_2010,lackenby_maximal_2013} for modern improvements), we may choose slopes $c_i\in P_i$ so that the resulting Dehn filled manifold $N(c_1,\ldots,c_n)$ is a closed, hyperbolic (in particular, aspherical) manifold.  Therefore,
\[
\wh{\pi_1N}/\overline{\llangle c_1,\ldots,c_n\rrangle}\cong\wh{\pi_1N}(c_1,\ldots,c_n)
\]
is a profinite $PD^3$ group.   Since $\overline{\llangle c_1,\ldots,c_n\rrangle}$  is generated by vertex stabilizers, $\wh{\pi_1N}(c_1,\ldots,c_n)$ acts on a profinite tree $\overline{\llangle c_1,\ldots,c_n\rrangle}\backslash T$ (see  \cite[Proposition 2.5]{zalesskii_subgroups_1988} or \cite[Proposition 4.1.1]{ribes_profinite_2017}) and still does not fix a vertex.   The edge stabilizers of the latter action are procyclic.  This contradicts Corollary \ref{cor: Good PDn groups}.
\end{proof}

\section{The JSJ decomposition}

In this section we show that, as well as the Kneser--Milnor decomposition, the JSJ decomposition is also determined by the profinite completion.  In order to avoid ambiguity, we start by stating the form of the JSJ decomposition we consider.  In a nutshell, it is the minimal decomposition along tori such that the complementary pieces are geometric.

\begin{definition}\label{defn: JSJ}
Let $M$ be a closed, orientable, irreducible 3-manifold which is not a torus bundle over the circle. Let $\mathcal{T}\subseteq M$ be an embedded disjoint union of essential tori such that the connected components of $M\smallsetminus\mathcal{T}$ are each geometric -- that is either Seifert fibred or admitting hyperbolic or Sol-geometry.   Such a union $\mathcal{T}$ with the smallest number of connected components is called the \emph{JSJ decomposition of $M$}.
\end{definition}

The existence of the JSJ decomposition follows from the work of Jaco--Shalen--Johannson together with Perelman's proof of the geometrization conjecture; see \cite[\S1.6, \S1.7]{aschenbrenner_manifold_2015} for details.  The tori are unique up to isotopy.   We follow Wilkes' elegant terminology \cite{wilkes_profinite_2018}, and use the term \emph{minor} to denote those components of $M\smallsetminus\mathcal{T}$ that are homeomorphic to the twisted interval bundle over the Klein bottle; the remaining components we call \emph{major}.    If two minor components are adjacent then their union is virtually a torus bundle over a circle, and so admits either Euclidean, Nil- or Sol-geometry, which contradicts the hypothesis that $\mathcal{T}$ was minimal.  Therefore, every edge adjoins at least one major vertex.

The submanifold $\mathcal{T}$ induces a graph-of-spaces decomposition of $M$, and hence a graph-of-groups decomposition of $\pi_1M$ and a profinite graph-of-groups decomposition of $\wh{\pi_1M}$ (see Example \ref{graph group completion}).  The Bass--Serre trees of the latter are denoted by $T_M$ and $\wh{T}_M$, respectively.  Crucially, these trees turn out to be acylindrical, in the sense of Definition \ref{defn: Acylindrical}.

\begin{proposition}\label{prop: JSJ properties}
For $M$ a closed, orientable, irreducible 3-manifold, the JSJ tree $T_M$ and the profinite JSJ tree $\wh{T}_M$ are both 4-acylindrical.
\end{proposition}
\begin{proof}
In \cite{wilton_profinite_2010} the authors showed that the corresponding decomposition of $\pi_1M$   is 4-acylindrical and fits into Example \ref{graph group completion}. In \cite{hamilton_separability_2013}, the authors showed with Hamilton that the corresponding profinite decomposition of $\wh{\pi_1M}$ is a 4-acylindrical injective graph of profinite groups (see also \cite[Lemma 4.5]{wilton_distinguishing_2017}).
\end{proof}

We are now ready to state our main theorem,

\begin{theorem}\label{thm: main}
If $M,M'$ are closed, orientable, irreducible 3-manifolds and
\[
f:\wh{\pi_1M}\stackrel{\cong}{\to}\wh{\pi_1{M'}}\]
is an isomorphism, then there is an $f$-equivariant isomorphism
\[
\phi:\wh{T}_M\to\wh{T}_{M'}
\]
of the corresponding profinite Bass--Serre trees.  In particular, the underlying graphs of the JSJ decompositions of $M$ and $M'$ are isomorphic, as are the profinite completions of the fundamental groups of the corresponding pieces.
\end{theorem}

Consider a vertex space $N$ of $M$.  The next three lemmas show that $\wh{\pi_1N}$ must act with a fixed point  on $\wh{T}_{M'}$.  We start with the hyperbolic case.

\begin{lemma}\label{lem: Hyperbolic fixed point}
Consider $N$ a compact, hyperbolic 3-manifold with (possibly empty) toral boundary. If $\wh{\pi_1N}$ acts acylindrically on a profinite tree $\wh{S}$ with abelian edge stabilizers then $\wh{\pi_1N}$ fixes a unique vertex.
\end{lemma}
\begin{proof}
If $N$ is closed then every abelian subgroup of $\wh{\pi_1N}$ is procyclic \cite[Theorem D]{wilton_distinguishing_2017} and the result follows from Corollary \ref{cor: No procyclic splitting}.

Suppose therefore that $N$ has non-empty toroidal boundary.   By Proposition \ref{prop: Cusp uniqueness} every edge stabilizer is either procyclic or conjugate into a peripheral subgroup, and therefore $\wh{\pi_1N}$ fixes a vertex by Lemma \ref{lem: non-splitting over cusp}.

Uniqueness follows from \cite[Corollary 2.9]{zalesskii_subgroups_1988} or  \cite[Corollary 4.1.6]{ribes_profinite_2017}, since $\wh{\pi_1N}$ is non-abelian and edge stabilizers are abelian.
\end{proof}

We move on to the major Seifert fibred case.

\begin{lemma}\label{lem: Major SF case}
Consider $N$ a compact, major Seifert fibred 3-manifold with (possibly empty) toral boundary. If $\wh{\pi_1N}$ acts acylindrically on a profinite tree $\wh{S}$ with abelian edge groups then $\wh{\pi_1N}$ fixes a unique vertex.
\end{lemma}
\begin{proof}
The  subgroups of $\wh{\pi_1N}$  which are  isomorphic to $\widehat{\mathbb{Z}}^2$  each fix a vertex by Lemma \ref{lem: Abelian acylindrical fixed point}. Thus the maximal  procyclic normal subgroup $C$ of $\wh{\pi_1N}$ fixes a vertex.

Suppose on the contrary $\wh{\pi_1N}$ does not fix a vertex. By \cite[Lemma 1.5]{zalesskii_profinite_1990} or \cite[Proposition 2.4.12]{ribes_profinite_2017}, there exists a unique  minimal $\wh{\pi_1N}$-invariant subtree $\wh{D}$ of $\wh{S}$ , which is infinite. Now by \cite[Theorem 2.12]{zalesskii_subgroups_1988} or \cite[Proposition 4.2.2]{ribes_profinite_2017}, $C$ acts trivially on $\wh{D}$, which  contradicts the acylindricity of the action.

Uniqueness again follows from \cite[Corollary 2.9]{zalesskii_subgroups_1988} or \cite[Corollary 4.1.6]{ribes_profinite_2017}.
\end{proof}

The case of minor Seifert-fibred vertex follows immediately from \cite[Theorem 5.2]{wilton_distinguishing_2017} and \cite[Corollary 2.9]{zalesskii_subgroups_1988} or \cite[Corollary 4.1.6]{ribes_profinite_2017}.

\begin{lemma}\label{lem: Minor SF case}
Consider $N$ a minor Seifert fibred 3-manifold. If $\wh{\pi_1N}$ acts acylindrically on a profinite tree $\wh{S}$ with abelian edge groups then $\wh{\pi_1N}$ fixes a unique vertex.
\end{lemma}

We next classify the fixed point sets of $\wh{\Z}^2$ subgroups of $\wh{\pi_1M}$.  First, we need an analysis of their normalizers in vertex stabilizers.  We start with the hyperbolic case, in which case normalizers coincide with centralizers.

\begin{lemma}\label{lem: Normalizer hyperbolic case}
Let $N$ be a compact, orientable, hyperbolic 3-manifold with toral boundary, and $H\cong\wh{\Z}^2$ a subgroup of $\wh{\pi_1N}$. Then $N_{\wh{\pi_1N}}(H)=C_{\wh{\pi_1N}}(H)$.
\end{lemma}
\begin{proof}
By \cite[Theorem 9.3]{wilton_distinguishing_2017}, $H$ is conjugate into the closure of a cusp subgroup, and by \cite[Lemma 4.5]{wilton_distinguishing_2017}, that cusp subgroup is malnormal. The result follows.
\end{proof}

We next treat the case of a major Seifert fibred manifold.

\begin{lemma}\label{lem: Normalizer major SF case}
Let $N$ be a compact, orientable, major Seifert-fibred 3-manifold with toral boundary, and $H\cong\wh{\Z}^2$ a subgroup of $\wh{\pi_1N}$ conjugate to the closure of the fundamental group of a boundary component. Then $N_{\wh{\pi_1N}}(H)=C_{\wh{\pi_1N}}(H)$.
\end{lemma}
\begin{proof}
The fundamental group $\pi_1N$ is torsion-free of the form
\[
1\to Z\to\pi_1N\to\pi_1O\to 1
\]
where $\pi_1O$ is a Fuchsian group and $Z$ is infinite cyclic (and not necessarily central).  Since Seifert-fibred 3-manifold groups are LERF \cite{scott_subgroups_1978,scott_correction_1985} we have a corresponding short exact sequence of profinite completions.
\[
1\to\wh Z\to\wh{\pi_1N}\overset{f}{\to}\wh{\pi_1O}\to 1
\]
Then $C_{\wh{\pi_1N}}(H)$ contains $\widehat Z$ and centralizes it. So $f(C_{\wh{\pi_1N}}(H))=C_{\wh{\pi_1O}}(f(H))$ and $f(N_{\wh{\pi_1N}}(H))=N_{\wh{\pi_1O}}(f(H))$.   Hence it suffices to  show that $C_{\wh{\pi_1O}}(f(H))=N_{\wh{\pi_1O}}(f(H))$. We may assume that $H$ is the closure of the fundamental group of a boundary component; then $f(H)$ is the closure of a peripheral infinite-cyclic subgroup $C$ of $\pi_1O$. Since Fuchsian groups are conjugacy separable \cite{fine_conjugacy_1990} we deduce that every-finite index subgroup of $\pi_1O$ is conjugacy separable. Then $\overline{C_{\pi_1O}(C)}=C_{\wh{\pi_1O}}(f(H))$ by  \cite[Corollary 12.3]{minasyan_hereditary_2012}, and  $\overline{N_{\pi_1O}(C)}=N_{\wh{\pi_1O}}(f(H))$ by \cite[Lemma 2.3 combined with Theorem 2.14]{chagas_hereditary_2013}. But $N_{\pi_1O}(C)=C_{\pi_1O}(C)$, so $C_{\wh{\pi_1O}}(f(H))=N_{\wh{\pi_1O}}(f(H))$ as required.
\end{proof}

Next, we classify the possible fixed subtrees for $\wh{\Z}^2$ subgroups of $\wh{\pi_1M}$.

\begin{lemma}\label{lem: Edge characterization}
Let $M$ be a closed, orientable, irreducible 3-manifold.  Consider the action of a subgroup $H\cong \wh{\Z}^2$ of $\wh{\pi_1M}$ on $\wh{T}_M$.  One of the following holds.
\begin{enumerate}[(i)]
\item \label{item: Unique fixed point}The fixed point set of $H$ is a vertex with Seifert-fibred stabilizer.
\item \label{item: Unique fixed edge}The fixed point set of $H$ consists of exactly one edge.
\item \label{item: Non-trivial normalizer} The fixed subtree of $H$ consists of exactly two edges; the central vertex has a minor Seifert fibred stabilizer, and the other two vertices are major.
\end{enumerate}
Furthermore, if the centralizer $C_{\wh{\pi_1M}}(H)$ is properly contained in the normalizer $N_{\wh{\pi_1M}}(H)$, then we are in case (\ref{item: Unique fixed point}) or case (\ref{item: Non-trivial normalizer}).
\end{lemma}
\begin{proof}
By Lemma \ref{lem: Abelian acylindrical fixed point} and Proposition \ref{prop: JSJ properties}, $H$ fixes a non-empty subtree. Recall that every edge of $\wh{T}_M$ adjoins a major vertex, and that every minor vertex of $\wh{T}_M$ adjoins exactly two edges.  If $H$ stabilizes an edge $e$ and an adjacent major vertex $v$ then, by Lemmas \ref{lem: Normalizer hyperbolic case} and \ref{lem: Normalizer major SF case}, $e$ is the unique edge incident at $v$ stabilized by $H$.  It follows that the fixed tree of $H$ is of one of the three claimed forms.

We now prove that, in case (\ref{item: Unique fixed edge}), $N_{\wh{\pi_1M}}(H)=C_{\wh{\pi_1M}}(H)$.   Indeed, $N_{\wh{\pi_1M}}(H)$ preserves the fixed subtree of $C_{\wh{\pi_1M}}(H)$, and so if $H$ fixes a unique edge, $N_{\wh{\pi_1M}}(H)$ is contained in an edge stabilizer, hence is abelian, and so $N_{\wh{\pi_1M}}(H)=C_{\wh{\pi_1M}}(H)$.
\end{proof}

We now have enough information to construct a map $\phi$.  To start with, it will only be a map of abstract, unoriented graphs.

\begin{lemma}\label{lem: Equivariant map}
Consider closed, orientable, irreducible 3-manifolds $M,M'$, and let $f:\wh{\pi_1M}\to\wh{\pi_1M'}$ be an isomorphism. Then there exists an $f$-equivariant morphism of graphs
\[
\phi:\wh{T}_M\to\wh{T}_{M'}~.
\]
Note that, here, we only claim that $\phi$ is a map of abstract, non-oriented graphs. This map may in principle send edges to either edges or vertices.
\end{lemma}
\begin{proof}
For brevity, we write $G=\wh{\pi_1M}$ and $\wh{S}=\wh{T}_{M'}$, and let $G$ act on $\wh{S}$ via $f$.  Let $e$ be an edge of $\wh{T}_M$ with stabilizer $G_e\cong\wh{\Z}^2$.  Let $u_1,u_2$ be the adjacent vertices of $\wh{T}_M$.  Lemmas \ref{lem: Hyperbolic fixed point}, \ref{lem: Major SF case} and \ref{lem: Minor SF case} together guarantee the existence of unique vertices $v_1,v_2$ of $\wh{S}$ such that $G_{u_i}\subseteq G_{v_i}$ for both $i$. 

We claim that $v_1$ and $v_2$ are either equal or adjacent.  Suppose therefore that $v_1,v_2$ are at distance greater than $2$ (possibly infinite).  Then $G_e$ stabilizes the geodesic $[v_1,v_2]$ (see \cite[Corollary 2.9]{zalesskii_subgroups_1988} or \cite[Corollary 4.1.6]{ribes_profinite_2017}) and therefore, by Lemma \ref{lem: Edge characterization}, $v_1,v_2$ are at distance precisely two, are both adjacent to a minor vertex $w$, and $C_G(G_e)$ is properly contained in $N_G(G_e)$.  Therefore, by Lemma \ref{lem: Edge characterization}, $e$ is adjacent to a minor vertex; without loss of generality, we may assume that $u_1$ is major, $G_{u_1}\subseteq G_{v_1}$, and that $u_2$ is minor and $G_{u_2}\subseteq G_{v_2}$. But $G_w$ also normalizes $G_e$, so $G_w\subseteq G_{u_2}\subseteq G_{v_2}$. This implies that $G_w$ stabilizes an edge, which is absurd because  edge stabilizers are abelian. Therefore, $v_1$ and $v_2$ are either equal or adjacent. If they are equal to a vertex $v$, we set $\phi(e)=v$.  If they are adjacent, we set $\phi(e)$ to be the image  of the unique edge joining them.  This completes the construction of the map $\phi:\wh{T}_M\to\wh{S}$, which is equivariant by construction.
\end{proof}

We are now ready to complete the proof of the main theorem.

\begin{proof}[Proof of Theorem \ref{thm: main}]
Applying Lemma \ref{lem: Equivariant map} twice, we obtain maps of graphs
\[
\phi:\wh{T}_M\to\wh{T}_{M'}~,~\psi:\wh{T}_{M'}\to\wh{T}_M~,
\]
where $\phi$ is $f$-equivariant and $\psi$ is $f^{-1}$-equivariant.  Equivariance implies that
\[
g\psi\circ\phi(x)=\psi\circ\phi(gx)
\]
for all $g\in\wh{\pi_1M}$ and $x\in\wh{T}_M$, whence the stabilizer of $x$ is contained in the stabilizer of $\psi\circ\phi(x)$. Since vertex-stabilizers stabilize unique vertices, it follows that $\psi\circ\phi$ is equal to the identity on vertices, and hence on the whole of $\wh{T}_M$.  In particular, $\phi$  and $\psi$ induce isomorphisms of the finite quotient graphs $\wh{\pi_1M}\backslash \wh{T}_M$ and $\wh{\pi_1M}'\backslash \wh{T}_M'$; we may therefore choose consistent orientations on these graphs, which lift to equivariant orientations on the profinite trees $\wh{T}_M$ and $\wh{T}_{M'}$, which are respected by $\phi$ and $\psi$; this also implies continuity of $\phi$ and $\psi$.
\end{proof}

\bibliographystyle{alpha}
%\bibliography{/Users/henry/Dropbox/library}

\end{document}